\documentclass[11pt]{amsart}
\usepackage{amsthm, amsfonts, amssymb, amsmath, bbold, enumitem, xcolor, url, stmaryrd}
\usepackage{avant}

\newcommand{\C}{\mathbb{C}}

\newcommand{\Q}{\mathbb{Q}}

\newcommand{\F}{\mathbb{F}}

\DeclareMathOperator{\Var}{Var}
\DeclareMathOperator{\Exp}{E}

\newtheorem{thm}{Theorem}[section]
\newtheorem{lem}[thm]{Lemma}

\theoremstyle{definition}
\newtheorem{ex}{Example}[section]

\begin{document}
\title[Variance of cubic curves over $\F_p$ and Jacobsthal sums]{Variance of point-counts for families of cubic curves over $\F_p$ and Jacobsthal sums}
\keywords{Number of points on cubic curves, variance, Jacobsthal sum.}
\subjclass[2020]{Primary: 11G20. Secondary: 11L10.}

\date{\today}
\author{Bogdan Nica}

\begin{abstract}  
We give explicit computations for the variance of the number of points along one-parameter families of cubic curves. We highlight evaluations of variances that involve Jacobsthal sums.
\end{abstract}

\address{\newline Department of Mathematical Sciences \newline Indiana University Indianapolis}
\email{bnica@iu.edu}

\maketitle
\setcounter{tocdepth}{1}
\tableofcontents

\section{Introduction}
Let $\F_p$ be the finite field with $p$ elements, where $p$ is a prime. Throughout, we assume that $p>3$.  

A general problem of great interest is that of counting the number of solutions in $\F_p$ to an equation of the form $y^2=f(x)$, where $f(x)$ is a polynomial map on $\F_p$; in a more geometric phraseology, the goal is to count the number of $\F_p$-points on the curve $y^2=f(x)$. For most polynomials, there is no explicit formula for the solution count. One way to circumvent this issue is to look for statistical insight. In this paper, we are interested in \emph{the variance of the number of points along one-parameter families of curves}.

Consider a family of curves $\mathcal{C}_\lambda: y^2=f_\lambda(x)$ parameterized by $\lambda\in \F_p$. Herein we exclusively consider families of cubic curves, in the sense that each $f_\lambda$ is a polynomial of degree $3$, but similar computations can be undoubtedly carried out for certain families of higher degree. Our main goal is to explicitly compute the variance of the point-count $\#\mathcal{C}=(\#\mathcal{C}_\lambda)_{\lambda\in \F_p}$, where
\[\#\mathcal{C}_\lambda=\#\big\{(x,y)\in \F_p\times \F_p: y^2=f_\lambda(x)\big\}\]
is the number of points on $\mathcal{C}_\lambda$.  

Let us recall that the \emph{variance} of a discrete random variable $X=(X_\lambda)_{\lambda\in \F_p}$ is defined as
\begin{align}\label{eq: var1}
\Var X=\Exp\Big( \big(X-\Exp(X)\big)^2\Big)
\end{align}
where the expected value, or the mean, is given by $\Exp(X)=\frac{1}{p}\sum_{\lambda\in \F_p} X_\lambda$.
Formula \eqref{eq: var1} captures the idea that the variance measures the deviation from the mean. A key property of the  variance, flowing out of formula \eqref{eq: var1}, is its translation invariance: $\Var (c+X)=\Var X$, where $c+X=(c+X_\lambda)_{\lambda\in \F_p}$. 

An alternate formula, oftentimes more convenient for computations, is  
\begin{align}\label{eq: var2}
\Var X=\Exp(X^2)-\Exp(X)^2.
\end{align}
The meaning of this second formula is that a variance computation usually combines two separate computations of expected values, $\Exp(X)$ and $\Exp(X^2)$. We note that, up to a factor of $p$, the two expected values $\Exp(X)$ and $\Exp(X^2)$ can be thought of as the first and the second moments of $X=(X_\lambda)_{\lambda\in \F_p}$, namely $\sum_{\lambda\in \F_p} X_\lambda$ and $\sum_{\lambda\in \F_p} X_\lambda^2$.

In what concerns our problem, that of explicitly computing the variance for the number of points along families of curves, the first step is to turn to quadratic character sums. Consider, once again, a family of curves $\mathcal{C}_\lambda: y^2=f_\lambda(x)$. For each $x\in \F_p$ the equation $y^2=f_\lambda(x)$ has $1+\sigma(f_\lambda(x))$ solutions, where $\sigma$ denotes the quadratic character on $\F_p$. Thus 
\begin{align*}
\#\mathcal{C}_\lambda=p+S_\lambda
\end{align*}
where
\begin{align}
S_\lambda=\sum_{x\in \F_p} \sigma\big(f_\lambda(x)\big).
\end{align}
It follows, by translation invariance, that
\begin{align}
\Var \#\mathcal{C}=\Var S
 \end{align}
 where $S=(S_\lambda)_{\lambda\in \F_p}$. This reformulation of the variance problem is more amenable to computations due to the various algebraic properties of the quadratic character $\sigma$--chief among them being its multiplicativity. The upshot, then, is the following: although our results are stated in terms of point-counts, our arguments are really about the variance of the quadratic character sums $(S_\lambda)_{\lambda\in \F_p}$.
 
 As a quick illustration of this idea, let us observe the following twisting invariance: for $d\in \F_p^*$, the family of curves $\mathcal{C}'_\lambda: dy^2=f_\lambda(x)$ has $\Var \#\mathcal{C}'=\Var \#\mathcal{C}$. Indeed, the twisted quadratic character sums satisfy the uniform signing rule $S'_\lambda=\sigma(d)S_\lambda$ for each $\lambda\in \F_p$. It follows that $\Exp(S')^2=\Exp(S)^2$ and $\Exp(S'^2)=\Exp(S^2)$, whence $\Var S'=\Var S$.

Evaluations for the first and second moments of the quadratic character sums $(S_\lambda)_{\lambda\in \F_p}$ associated to various families of cubics have been considered before. Without attempting to be exhaustive, we mention the following families of cubics:
 \begin{itemize}
 \item[(i)] the families $\{y^2=x^3+\lambda x: \lambda \in \F_p\}$ and $\{y^2=x^3+\lambda : \lambda \in \F_p\}$ (Jacobsthal \cite{J0, J});
 \item[(ii)]  the two-parameter Weierstrass family $\{y^2=x^3+\lambda x+\lambda' : \lambda ,\lambda' \in \F_p\}$ (Birch \cite{B});
\item[(iii)]  the one-parameter Weierstrass families $\{y^2=x^3+\lambda x+c: \lambda \in \F_p\}$, for fixed $c\in\F_p^*$, respectively $\{y^2=x^3+bx+\lambda : \lambda \in \F_p\}$ for fixed $b\in \F_p^*$ (He--McLaughlin \cite{HM}); 
 \item[(iv)] the Legendre family $\{y^2=x(x-1)(x-\lambda): \lambda\in \F_p\}$ (Hopf \cite{H}, Yamauchi \cite{Y});
  \item[(v)] the families $\{y^2=x^3+\lambda  x^2+1: \lambda \in \F_p\}$, $\{y^2=x^3-\lambda^2(x-1): \lambda \in \F_p\}$, and $\{y^2=x^3-3x+1+\lambda (x^2-x): \lambda \in \F_p\}$ (Miller \cite{Mil1, Mil2}).
 \end{itemize}
 
Further first and second moments have been worked out by Steven J. Miller and his collaborators. Miller's original motivation is a certain negative bias phenomenon which impacts the rank of an elliptic curve over $\Q$, cf. \cite{Mil1}.
 
Somewhat abusively, such families of cubics are often referred to as families of elliptic curves. The generic cubic in each family is indeed separable whence elliptic, but there may be exceptional non-elliptic cubics in each family. By restricting the parameter range, one can impose that the families of cubics consist of elliptic curves only. Thus, up to minor adjustments, it is possible to deduce explicit computations of variances over families of genuine elliptic curves. As far as we can tell, the adjusted results offer more wrinkles than wisdom; we chose to not pursue them in detail. 

Our results are inspired by the evaluations of first and second moments mentioned above. Herein we do three things. Firstly, on the technical side, we generalize some of them. Secondly, on the conceptual side, we introduce and pursue the variance viewpoint\footnote{Other works have considered how point-counts vary over families of cubic curves, cf. the titles of \cite{B} and \cite{Mil2}. But the variance itself has not been considered before, to the best of our knowledge.}. This way of framing first- and second-moment calculations has an attractive statistical interpretation, and several stability features. Thirdly, we highlight variance evaluations that end up involving certain quadratic character sums known as \emph{Jacobsthal sums}. By using known evaluations of Jacobsthal sums we achieve \emph{interesting} explicit evaluations of variances for several families of cubics. The usefulness of Jacobsthal sums in evaluating second moments--or variances, in this paper's perspective--has been hitherto untapped.

\medskip
\section{Results}
We now state our main results, in which we obtain explicit computations for the variance of the number of points along certain one-parameter families of cubic curves. The first three results concern the general cubic $y^2=x^3+ax^2+bx+c$, viewed as a one-parameter family by fixing two of the coefficients $a,b,c$ and letting the third vary over $\F_p$.

\begin{thm}\label{thm: C}
Fix $a,b\in \F_p$ and consider the family of curves $\mathcal{C}_\lambda: y^2=x^3+a x^2+bx+\lambda$ parameterized by $\lambda\in \F_p$. Then
\begin{align*}
\Var \#\mathcal{C}=\begin{cases}
p-1-\sigma(-3)-\sigma(a^2-3b) & \textrm{ if } a^2\neq 3b,\\
\big(1+\sigma(-3)\big)(p-1) & \textrm{ if } a^2=3b.
\end{cases}
\end{align*}
\end{thm}

\begin{thm} \label{thm: B}
Fix $a,c\in \F_p$ and consider the family of curves $\mathcal{C}_\lambda: y^2=x^3+a x^2+\lambda x+c$ parameterized by $\lambda\in \F_p$.
\begin{itemize}
\item[(i)] If $c=0$ then
\begin{align*}
\Var \#\mathcal{C}&=\begin{cases}
\big(1+\sigma(-1)\big)(p-1) & \textrm{ if } a=0,\\
p-2-\sigma(-1) & \textrm{ if } a\neq 0.
\end{cases}
\end{align*}
\item[(ii)] If $c\neq 0$ then
\begin{align*}
\Var \#\mathcal{C} =p-1-\sigma(-1)-n_3+\sigma(c)\sum_{x\in \F_p} \sigma(x^3+ax^2-4c)
\end{align*}
where $n_3$ is the number of solutions in $\F_p$ to the equation $x^3+ax^2-4c=0$.
\end{itemize}
\end{thm}

\begin{thm} \label{thm: A}
Fix $b,c\in \F_p$, not both $0$, and consider the family of curves $\mathcal{C}_\lambda: y^2=x^3+\lambda x^2+bx+c$ parameterized by $\lambda\in \F_p$.
\begin{itemize}
\item[(i)] If $c=0$ then
\begin{align*}
\Var \#\mathcal{C}=p-1-\sigma(b)-\frac{1}{p}.
\end{align*}
\item[(ii)] If $c\neq 0$ then
\begin{align*}
\Var \#\mathcal{C} =p-1-n_3-\frac{1}{p}+\sum_{x\in \F_p} \sigma\big(4cx^3+(bx+c)^2\big)
\end{align*}
where $n_3$ is the number of solutions in $\F_p$ to the equation $x^3-bx-2c=0$.
\end{itemize}
\end{thm}

We prove these results, and we exemplify them, in Sections~\ref{sec: M}, ~\ref{sec: O}, ~\ref{sec: N} respectively. Theorem~\ref{thm: C} turns out to be the simplest. Theorems~\ref{thm: B} and ~\ref{thm: A} are more technically involved. Visibly, in each one there still is a residual quadratic character sum. We choose examples in which we can evaluate them by means of Jacobsthal sums. 

Our last result is of a different nature. We take the general depressed cubic $y^2=x^3+bx+c$, and we perturb it into the family $y^2=x^3+bx+c+\lambda (x^2-x)$. We compute the variance of the point-count for certain pairs $(b,c)$.

\begin{thm}\label{thm: D}
Fix $b,c\in \F_p$ and consider the family of curves $\mathcal{C}_\lambda : y^2=x^3+bx+c+\lambda (x^2-x)$ parameterized by $\lambda\in \F_p$. Then 
\begin{align*}
\Var \#\mathcal{C}=p-2-\frac{1}{p}+\Sigma(b,c)
\end{align*}
where the value of $\Sigma(b,c)$ is tabulated below for certain choices of $b$ and $c$.
 \begin{table}[h]\renewcommand{\arraystretch}{1.5}
 \begin{tabular}{|c | c|} 
 \hline
 $(b,c)$ &\quad $\Sigma(b,c)$  \\
 \hline\hline
$(0,0)$, $(0,-1)$ \quad&\quad $-1-\sigma(-1)$ \\
\quad  $(-2,0)$, $(-2,1)$ \quad&\quad $-1-\sigma(-1)+\sigma(2) \varphi_2(1)$ \\
\quad $(1,0)$, $(1,-2)$ \quad&\quad $-1-\sigma(2)+\varphi_2(1)$ \\
 \quad $(-1/2,0)$, $(-1/2,-1/2)$ \quad&\quad $-1-\sigma(2)+\sigma(2) \varphi_2(1)$ \\
  \hline
   $(-1,0)$ \quad&\quad $0$ \\
  $(-3,1)$ \quad&\quad $-2\big(1+\sigma(-1)+\sigma(-3)\big)$ \\
    $(1,-1)$ \quad&\quad $-n_4-2\sigma(-1)+\sigma(-1)\varrho(2)+\sigma(2)\varphi_2(1)$ \quad \\
   \hline   
 \end{tabular} 
\end{table}

In the last entry, $n_4$ denotes the number of solutions in $\F_p$ to the quartic equation $x^4-2x^3+x^2-2x+1=0$, and 
\[\varrho(2)=\sum_{x\in \F_p} \sigma(x^3+x^2+2x).\]
\end{thm}

The case $(b,c)=(-3,1)$ recovers first- and second-moment computations by Miller \cite[Sec.10.3]{Mil1}. In the case $(b,c)=(-1,0)$, we have $x^3+bx+c+\lambda (x^2-x)=x^3-x+\lambda(x^2-x)=x(x-1)(x+\lambda+1)$. Up to a slight reparameterization, we recover the Legendre family $\{y^2=x(x-1)(x-\lambda): \lambda\in \F_p\}$ previously considered by Hopf \cite{H} and Yamauchi \cite{Y}. 

The quantity $\varphi_2(1)$ that appears in several places in the table above is an instance of a Jacobsthal sum; we elucidate its meaning, and its value, in the next section. We find the computations in the second, third, and fourth rows of the table as the most satisfactory. The computation in the case $(b,c)=(1,-1)$ is not fully explicit, as we do not have a formula for the term $\varrho(2)$. We work out Theorem~\ref{thm: D} in Section~\ref{sec: S}.

Throughout, our methods involve quadratic character sums. For comparison, Birch \cite{B} and He--McLaughlin \cite{HM} use the viewpoint of exponential sums. Most of our results are intractable through their exponential viewpoint. But Birch \cite{B} and He--McLaughlin \cite{HM} also obtain higher moment computations for certain families of cubics--specifically, moments of order $4$, $6$, and $8$ in \cite{B}, respectively moments of order $3$ in \cite{HM}. We are unable to compute higher moments via quadratic character sums.

\medskip
\section{Preliminaries}
\subsection{The quadratic character} The quadratic character $\sigma:\F_p\to \C^*$ is defined as follows. For $a\neq 0$, set $\sigma(a)=1$ if $a\in (\F_p^*)^2$, that is to say, if $a$ is a square in $\F^*_p$, and set $\sigma(a)=-1$ if $a\not\in (\F_p^*)^2$, that is to say, if $a$ is not a square in $\F^*_p$. For $a=0$, set $\sigma(0)=0$.

Several properties of the quadratic character will frequently come up. Firstly, $\sigma$ is multiplicative: $\sigma(ab)=\sigma(a)\sigma(b)$ for all $a,b\in \F_p$. Secondly, $\sigma$ counts solutions to quadratic equations: the equation $y^2=c$ has $1+\sigma(c)$ solutions in $\F_p$; more generally, the equation $ay^2+by+c=0$ has $1+\sigma(b^2-4ac)$ solutions in $\F_p$ whenever $a\neq 0$. Thirdly, the following summation formulas hold:
\begin{align}\label{eq: qjs1}
\sum_{x\in \F_p} \sigma(x)= 0,
\end{align}
and
\begin{align}\label{eq: qjs2}
\sum_{x\in \F_p} \sigma(x^2+c)= p\llbracket c=0\rrbracket-1,
\end{align}
\begin{align}\label{eq: qjs3}
\sum_{x\in \F_p} \sigma(x+c)\sigma(x+c')= p\llbracket c=c'\rrbracket-1.
\end{align}
Here and in what follows we use the Iverson bracket notation: if $P$ is a statement, then $\llbracket P\rrbracket=1$ if $P$ is true, respectively $\llbracket P\rrbracket=0$ if $P$ is false.

\subsection{Cubic Jacobsthal sums}
There are two families of Jacobsthal sums that will come up in what follows:
\begin{align}
\varphi_2(c)&=\sum_{x\in \F_p} \sigma(x^3+cx), \\
\psi_3(c)&=\sum_{x\in \F_p} \sigma(x^3+c)
\end{align}
where $c\in \F^*_p$. As the indexing suggests, these quadratic character sums with cubic arguments have higher degree analogues. They will not be needed in this paper; we refer the interested reader to the recent monograph \cite{N} for much more on properties and applications of Jacobsthal sums.

The cubic Jacobsthal sums $\varphi_2(c)$ and $\psi_3(c)$ can be evaluated in terms of decompositions of $p$ into sums of squares. The following two lemmas collect results from \cite[Chapter 3]{N}, in a form that is convenient for our purposes.

\begin{lem}\label{lem: J1} The following hold for $c\in \F_p^*$:
\begin{itemize}
\item[(i)] if $p\equiv 3$ mod $4$ then $\varphi_2(c)=0$;
\item[(ii)] if $p\equiv 1$ mod $4$ then
\[\varphi_2(c)=\begin{cases}
\pm 2 A_2 & \textrm{ if } c\in (\F_p^*)^2 \\
\pm 2 B_2 & \textrm{ if } c\notin (\F_p^*)^2
\end{cases}
\]
\end{itemize}
where $p=A_2^2+B_2^2$, with the convention that $A_2\equiv -1$ mod $4$. In the case $c\in (\F_p^*)^2$, say $c=s^2$, the sign is determined as $\varphi_2(s^2)=\sigma(s)\cdot 2A_2$. In particular $\varphi_2(1)=2A_2$.
\end{lem}

\begin{lem}\label{lem: J2} The following hold for $c\in \F_p^*$:
\begin{itemize}
\item[(i)] if $p\equiv 2$ mod $3$ then $\psi_3(c)=0$;
\item[(ii)] if $p\equiv 1$ mod $3$ then
\[\sigma(c)\psi_3(c)=\begin{cases}
2A_3 & \textrm{ if } c\in (\F_p^*)^3 \\
-A_3\pm 3B_3 & \textrm{ if } c\notin (\F_p^*)^3
\end{cases}
\]
\end{itemize}
where $p=A_3^2+3B_3^2$, with the convention that $A_3\equiv -1$ mod $3$. In particular $\psi_3(1)=2A_3$.
\end{lem}

\medskip
\section{The family $\mathcal{C}_\lambda: y^2=x^3+a x^2+bx+\lambda$}\label{sec: M}
\begin{proof}[Proof of Theorem~\ref{thm: C}] For $\lambda\in\F_p$, set
\begin{align*}
S_\lambda= \sum_{x\in \F_p} \sigma(x^3+ax^2+bx+\lambda).
\end{align*}
Put $r(x)=x^3+ax^2+bx$, so
\[S_\lambda= \sum_{x\in \F_p} \sigma(\lambda+r(x)).\]

We set out to compute the variance of $(S_\lambda)_{\lambda\in \F_p}$. Using \eqref{eq: qjs1}, we see that $\Exp (S)=0$. Thus $\Var S=\Exp (S^2)$. Now, using \eqref{eq: qjs3}, we have
\begin{align*}
\sum_{\lambda\in \F_p} S_\lambda^2&= \sum_{x,y\in \F_p} \sum_{\lambda\in \F_p} \sigma(\lambda+r(x)) \sigma(\lambda+r(y))\\
&= \sum_{x,y\in \F_p} \Big(p\llbracket r(x)=r(y)\rrbracket-1\Big)\\
&= p\sum_{x\neq y\in \F_p}\llbracket r(x)=r(y)\rrbracket.
\end{align*}
In the last step, the diagonal contribution $p\sum_{x=y\in \F_p}\llbracket f(x)=f(y)\rrbracket=p^2$ is cancelled by $\sum_{x,y\in \F_p} 1=p^2$. For $x\neq y$ we have $\llbracket r(x)=r(y)\rrbracket=\llbracket q(x,y)=0\rrbracket$ where \[q(x,y)= \frac{r(x)-r(y)}{x-y}=x^2+xy+y^2+a(x+y)+b.\]
Thus
\begin{align*}
\Var S&=\frac{1}{p}\sum_{\lambda\in \F_p} S_\lambda^2=\sum_{x\neq y\in \F_p}\llbracket  q(x,y)=0\rrbracket\\
&=\sum_{x, y\in \F_p}\llbracket q(x,y)=0\rrbracket-\sum_{x\in \F_p}\llbracket q(x,x)=0\rrbracket.
\end{align*}
The latter sum, $\sum_{x\in \F_p}\llbracket q(x,x)=0\rrbracket$, counts the number of solutions to the equation $q(x,x)=0$, that is to say, $3x^2+2ax+b=0$; hence it evaluates as $1+\sigma(a^2-3b)$. Likewise, for each $x\in \F_p$ the sum $\sum_{y\in \F_p}\llbracket q(x,y)=0\rrbracket$ counts the number of solutions to the quadratic equation $y^2+(x+a)y+x^2+ax+b=0$. Here the discriminant is $(x+a)^2-4(x^2+ax+b)=-3x^2-2ax+a^2-4b$, so the quadratic equation has $1+\sigma(-3x^2-2ax+a^2-4b)$ solutions. Thus
\begin{align*}
\sum_{x, y\in \F_p}\llbracket q(x,y)=0\rrbracket&=\sum_{x\in \F_p} \big(1+\sigma(-3x^2-2ax+a^2-4b)\big)\\
&=p+\sigma(-3)\sum_{x\in \F_p} \sigma\big(x^2+2ax-3(a^2-4b)\big)\\
&=p+\sigma(-3)\sum_{x\in \F_p} \sigma\big((x+a)^2-4(a^2-3b)\big)\\
&=p+\sigma(-3) \big(p\llbracket a^2-3b=0\rrbracket-1\big).
\end{align*}
We have effected a change of variables $x:=x/3$ in the second step, and we have used \eqref{eq: qjs2} in the last step. 

Combining these calculations, we get
\begin{align*}
\Var S&=p+\sigma(-3) \big(p\llbracket a^2-3b=0\rrbracket-1\big)-\big(1+\sigma(a^2-3b)\big)\\
&=\begin{cases}
p-1-\sigma(-3)-\sigma(a^2-3b) & \textrm{ if } a^2\neq 3b,\\
\big(1+\sigma(-3)\big)(p-1) & \textrm{ if } a^2=3b.
\end{cases}
\end{align*}
This completes the proof of Theorem~\ref{thm: C}.
\end{proof}

\begin{ex}
Fix $b\in \F_p^*$. For the family of curves $\mathcal{C}_\lambda: y^2=x^3+bx+\lambda$, the variance of the point-count is
\begin{align}\label{eq: ca0}
\Var \#\mathcal{C}=p-1-\sigma(-3)-\sigma(-3b).
\end{align}
We are thus recovering, in a variance format, first- and second-moment computations by He and McLaughlin \cite[Thms.3,5]{HM}. Their methods are, however, different--based on exponential sums.
\end{ex}

\begin{ex} Fix $b\in \F_p^*$. We compute the variance of the number of points over the family $\mathcal{C}_\lambda :y^2=x^3+\lambda^2(bx+1)$ parameterized by $\lambda\in \F_p$. Miller \cite[Sec.14.1]{Mil1} has considered this family in the case $b=-1$. Although the result of Theorem~\ref{thm: C} does not directly address this family, we will see that one can still establish a link. 

As usual, $\Var \#\mathcal{C}=\Var T$ where $T=(T_\lambda)_{\lambda\in \F_p}$ is given by
\begin{align*}
T_\lambda= \sum_{x\in \F_p} \sigma\big(x^3+\lambda^2(bx+1)\big).
\end{align*}
We evaluate the expected value of $T$ directly. We have
\begin{align*}
\sum_{\lambda\in \F_p}T_\lambda= \sum_{x\in \F_p} \sum_{\lambda\in \F_p} \sigma\big((bx+1)\lambda^2+x^2\big)
\end{align*}
and the inner sum can be evaluated as follows: when $x=0$, it equals $p-1$; when $bx+1=0$, it equals $p$; when $x\neq 0$ and $bx+1\neq 0$, it equals $-\sigma(bx+1)$ thanks to ~\eqref{eq: qjs2}. Overall:
\begin{align*}
\sum_{\lambda\in \F_p}T_\lambda
&=(p-1)+p-\sum_{x\neq 0,\; bx+1\neq 0} \sigma(bx+1)\\
&=2p-1-(-\sigma(1))=2p
\end{align*}
where, on the way, we have used ~\eqref{eq: qjs1} once again. Therefore $\Exp (T)=2$.

We now turn to the expected value of $T^2$. Note that $T_0=0$. For $\lambda\in \F^*_p$, the change of variable $x:=\lambda x$ gives
\begin{align*}
T_\lambda=\sigma(\lambda)\sum_{x\in \F_p} \sigma(x^3+bx+\lambda^{-1}).
\end{align*}
Hence $T_{\lambda^{-1}}=\sigma(\lambda)S_\lambda$, where 
\begin{align*}
S_\lambda=\sum_{x\in \F_p} \sigma(x^3+b x+\lambda).
\end{align*}
Therefore
\begin{align*}
\Exp (T^2)=\frac{1}{p} \sum_{\lambda\in \F_p^*} T_\lambda^2=\frac{1}{p} \sum_{\lambda\in \F_p^*} S_\lambda^2=\Exp (S^2)-\frac{1}{p}S_0^2.
\end{align*}
Now, $S_0=\varphi_2(b)$. We also know from the proof of Theorem~\ref{thm: C} and \eqref{eq: ca0} above, that $\Exp (S^2)=p-1-\sigma(-3)-\sigma(-3b)$. We deduce that 
\begin{align*}
\Exp (T^2)=p-1-\sigma(-3)-\sigma(-3b)-\varphi_2(b)^2/p.
\end{align*}
All in all, we conclude that
\begin{align}
\Var \#\mathcal{C}=p-5-\sigma(-3)-\sigma(-3b)-\varphi_2(b)^2/p.
\end{align}

We can be even more explicit, by appealing to Lemma~\ref{lem: J1}:
\begin{align}
\Var \#\mathcal{C}=\begin{cases}
p-5-\sigma(-3)-\sigma(-3b) & \textrm{ if } p\equiv 3 \textrm{ mod } 4;\\
p-5-2\sigma(-3)-4A_2^2/p & \textrm{ if } p\equiv 1 \textrm{ mod } 4, b\in (\F_p^*)^2,\\
p-5-4B_2^2/p & \textrm{ if } p\equiv 1 \textrm{ mod } 4, b\not\in (\F_p^*)^2.
\end{cases}
\end{align}
\end{ex}

\medskip
\section{The family $\mathcal{C}_\lambda: y^2=x^3+a x^2+\lambda x+c$}\label{sec: O}
\begin{proof}[Proof of Theorem~\ref{thm: B}] For $\lambda\in\F_p$, set
\begin{align*}
S_\lambda= \sum_{x\in \F_p} \sigma(x^3+ax^2+\lambda x+c).
\end{align*}
We aim to compute the variance of $S=(S_\lambda)_{\lambda\in \F_p}$. For $x\in \F_p^*$, put 
\begin{align}
r(x)=\frac{x^3+ax^2+c}{x}
\end{align} 
so that
\[S_\lambda= \sigma(c)+\sum_{x\in \F^*_p}\sigma(x) \sigma(\lambda+r(x)).\]
Thanks to translation invariance, the variance of $S=(S_\lambda)_{\lambda\in \F_p}$ equals that of $T=(T_\lambda)_{\lambda\in \F_p}$, where 
\begin{align}
T_\lambda=\sum_{x\in \F^*_p}\sigma(x) \sigma(\lambda+r(x)).
\end{align}

By using \eqref{eq: qjs1} we see that $\Exp (T)=0$:
\[\sum_{\lambda\in \F_p} T_\lambda=\sum_{x\in \F^*_p}\sigma(x) \sum_{\lambda\in \F_p} \sigma(\lambda+r(x))=0.\]
 Thus $\Var T=\Exp (T^2)$. Now, using \eqref{eq: qjs3}, we see that
\begin{align*}
\sum_{\lambda\in \F_p} T_\lambda^2&=\sum_{x,y\in \F^*_p} \sigma(xy)\sum_{\lambda\in \F_p} \sigma(\lambda+r(x))\sigma(\lambda+r(y))\\
&=\sum_{x,y\in \F^*_p} \sigma(xy) \Big(p\llbracket r(x)=r(y)\rrbracket-1\Big).
\end{align*}
We note that $\sum_{x,y\in \F^*_p} \sigma(xy)=\big(\sum_{x\in \F^*_p} \sigma(x)\big)^2=0$. Hence
\begin{align*}
\Var T&=\frac{1}{p}\sum_{\lambda\in \F_p} T_\lambda^2= \sum_{x,y\in \F^*_p} \sigma(xy) \llbracket r(x)=r(y)\rrbracket\\
&=p-1+ \sum_{x\neq y\in \F^*_p} \sigma(xy) \llbracket r(x)=r(y)\rrbracket.
\end{align*}
Let $\Sigma$ denote the latter double sum. Whenever $x\neq y\in \F_p^*$, we have $\llbracket r(x)=r(y)\rrbracket=\llbracket q(x,y)=0\rrbracket$ where \[q(x,y)= xy\cdot\frac{r(x)-r(y)}{x-y}=xy(x+y)+axy-c.\]
Thus
\begin{align*}
\Sigma&=\sum_{x\neq y\in \F^*_p}\sigma(xy) \llbracket  q(x,y)=0\rrbracket\\
&=\sum_{x,y\in \F^*_p}\sigma(xy)\llbracket  q(x,y)=0\rrbracket-\sum_{x\in \F^*_p}\llbracket  q(x,x)=0\rrbracket
\end{align*}
by reinstating the diagonal in the main double sum. 

The equation $q(x,x)=0$ amounts to $2x^3+ax^2-c=0$ which, upon scaling $x:=x/2$, turns into $x^3+ax^2-4c=0$. Thus
$\sum_{x\in \F^*_p}\llbracket  q(x,x)=0\rrbracket$ evaluates as $n_3^*=\#\{x\in \F_p^*: x^3+ax^2-4c=0\}$. In the main double sum, we make the change of variable $x:=x/y$ for each $y\in \F_p^*$:
\begin{align*}
\sum_{x,y\in \F^*_p}\sigma(xy)\llbracket q(x,y)=0\rrbracket=\sum_{x\in \F_p^*} \sigma(x)\sum_{y\in \F^*_p} \llbracket q(x/y,y)=0\rrbracket.
\end{align*}
Now, for each $x\in \F_p^*$, $q(x/y,y)=0$ amounts to the quadratic equation $xy^2+(ax-c)y+x^2=0$. This has $1+\sigma\big((ax-c)^2-4x^3\big)$ solutions $y\in \F_p$, and each one is automatically non-zero. Thus
\begin{align*}
\sum_{x\in \F_p^*} \sigma(x)\sum_{y\in \F^*_p} \llbracket q(x/y,y)=0\rrbracket&=\sum_{x\in \F_p^*} \sigma(x) \Big(1+\sigma\big((ax-c)^2-4x^3\big)\Big)\\
&=\sum_{x\in \F^*_p} \sigma\big(x(ax-c)^2-4x^4\big)\\
&=\sum_{x\in \F^*_p} \sigma\big(x(cx-a)^2-4\big)
\end{align*}
after an inversive change of variable $x:=1/x$ in the last step.

Putting everything together, we get
\begin{align}\label{eq: bw}
\Var T=p-1+\Sigma=p-1-n^*_3+\sum_{x\in \F^*_p} \sigma\big(x(cx-a)^2-4\big).
\end{align}

(i) Assume that $c=0$. Then
\[n^*_3=\#\{x\in \F_p^*: x^3+ax^2=0\}=\begin{cases}
0 & \textrm{ if } a=0,\\
1 & \textrm{ if } a\neq 0.
\end{cases}\]
Furthermore, by \eqref{eq: qjs1} 
\begin{align*}
\sum_{x\in \F^*_p} \sigma(a^2x-4)&=
\begin{cases}
\sigma(-1)(p-1) & \textrm{ if } a=0,\\
-\sigma(-1) & \textrm{ if } a\neq 0.
\end{cases}
\end{align*}
In this case, from \eqref{eq: bw} we conclude that
\begin{align}
\Var T&=\begin{cases}
\big(1+\sigma(-1)\big)(p-1) & \textrm{ if } a=0,\\
p-2-\sigma(-1) & \textrm{ if } a\neq 0.
\end{cases}
\end{align}

(ii) Assume that $c\neq 0$. Then $n_3^*=\#\{x\in \F_p^*: x^3+ax^2-4c=0\}$ agrees with $n_3=\#\{x\in \F_p: x^3+ax^2-4c=0\}$ due to solutions being necessarily non-zero. Furthermore, the change of variables $x:=(x+a)/c$ gives
\begin{align*}
\sum_{x\in \F^*_p} \sigma\big(x(cx-a)^2-4\big)&= \sum_{x\in \F^*_p} \sigma((x+a)x^2/c-4)\\
&=-\sigma(-1)+\sigma(c)\sum_{x\in \F_p} \sigma(x^3+ax^2-4c).
\end{align*}
All in all, \eqref{eq: bw} becomes
\begin{align}
\Var T=p-1-\sigma(-1)-n_3+\sigma(c)\sum_{x\in \F_p} \sigma(x^3+ax^2-4c).
\end{align}
This completes the proof of Theorem~\ref{thm: B}.
\end{proof}

\begin{ex}
Fix $c\in \F_p^*$. For the family of curves $\mathcal{C}_\lambda : y^2=x^3+\lambda x+c$, the variance of the point-count is given by
\begin{align}\label{eq: a0b}
\Var \#\mathcal{C}=p-1-\sigma(-1)-n_3+\sigma(c) \sum_{x\in \F_p} \sigma(x^3-4c)
\end{align}
where $n_3$ is the number of solutions in $\F_p$ to $x^3=4c$. In the case $c=1$, formula \eqref{eq: a0b} resolves a difficulty raised by Miller \cite[Rem.2.2]{Mil2}.

The right-most sum in \eqref{eq: a0b} is $\psi_3(-4c)$. We can use Lemma~\ref{lem: J2} to give a more explicit evaluation. When $p\equiv 2$ mod $3$, we have $\psi_3(-4c)=0$ and $n_3=1$ for each $c\in \F_p^*$, so
\begin{align}
\Var \#\mathcal{C}=p-2-\sigma(-1).
\end{align}
This evaluation essentially recovers first- and second-moment computations of He and McLaughlin \cite[Thms.3,4]{HM}. But the general formula \eqref{eq: a0b} seems unachievable by their method--namely, exponential sums.  In particular, they are unable to handle the much more interesting case of \eqref{eq: a0b}, when $p\equiv 1$ mod $3$. 

When $p\equiv 1$ mod $3$, there are two cases according to whether $4c$ is, or is not, a cube in $\F_p^*$. We have, respectively, that $n_3=3$ or $n_3=0$, and
\begin{align}
\Var \#\mathcal{C}=\begin{cases}
p-4-\sigma(-1)+\sigma(-1)\cdot 2A_3& \textrm{ if } 4c\in (\F_p^*)^3, \\
p-1-\sigma(-1)+\sigma(-1)(-A_3\pm 3B_3) & \textrm{ if } 4c\notin (\F_p^*)^3.
\end{cases}
\end{align}
\end{ex}

\begin{ex}
For the family of curves $\mathcal{C}_\lambda : y^2=x^3+6x^2+\lambda x+4$, the variance of the point-count is given by
\begin{align*}
\Var \#\mathcal{C}=p-1-\sigma(-1)-n_3+ \sum_{x\in \F_p} \sigma(x^3+6x^2-16)
\end{align*}
where $n_3$ is the number of solutions in $\F_p$ to $x^3+6x^2-16=0$. The cubic $x^3+6x^2-16$ factors as $(x+2)(x^2+4x-8)$. Firstly, we see that $n_3=2+\sigma(3)$. Secondly, we can handle the above quadratic character sum by a change of variable $x:=x-2$. This yields 
\begin{align*}
\sum_{x\in \F_p} \sigma(x^3+6x^2-16)=\sum_{x\in \F_p} \sigma\big(x(x^2-12)\big)=\varphi_2(-12).
\end{align*}
Summarizing, we have
\begin{align}
\Var \#\mathcal{C}=p-3-\sigma(-1)-\sigma(3)+\varphi_2(-12).
\end{align}
We can be more explicit, by appealing to Lemma~\ref{lem: J1} and recalling that
\begin{align*}
\sigma(-1)=
\begin{cases} 
1 & \textrm{ if } p\equiv 1 \textrm{ mod } 4,\\
-1 & \textrm{ if } p\equiv 3 \textrm{ mod } 4;
\end{cases}, \quad \sigma(-3)=
\begin{cases} 
1 & \textrm{ if } p\equiv 1 \textrm{ mod } 3,\\
-1 & \textrm{ if } p\equiv 2 \textrm{ mod } 3.
\end{cases}
\end{align*}
The outcomes are as follows:
\begin{align}
\Var \#\mathcal{C}=
\begin{cases}
p-5\pm 2A_2 & \textrm{ if } p\equiv 1 \textrm{ mod } 12,\\
p-3\pm 2B_2 & \textrm{ if } p\equiv 5 \textrm{ mod } 12,\\
p-1 & \textrm{ if } p\equiv 7 \textrm{ mod } 12,\\
p-3 & \textrm{ if } p\equiv 11 \textrm{ mod } 12.
\end{cases}
\end{align}
\end{ex}

\medskip
\section{The family $\mathcal{C}_\lambda: y^2=x^3+\lambda x^2+bx+c$}\label{sec: N}
\begin{proof}[Proof of Theorem~\ref{thm: A}] We compute the variance of $S=(S_\lambda)_{\lambda\in \F_p}$, where
\begin{align*}
S_\lambda= \sum_{x\in \F_p} \sigma(x^3+\lambda x^2+bx+c).
\end{align*}
 For $x\in \F_p^*$, put 
\begin{align}
r(x)=\frac{x^3+bx+c}{x^2}
\end{align} 
so that
\[S_\lambda= \sigma(c)+\sum_{x\in \F^*_p} \sigma(\lambda+r(x)).\]
Thanks to translation invariance, the variance of $S=(S_\lambda)_{\lambda\in \F_p}$ equals that of $T=(T_\lambda)_{\lambda\in \F_p}$, where 
\begin{align}
T_\lambda=\sum_{x\in \F^*_p} \sigma(\lambda+r(x)).
\end{align}
By using \eqref{eq: qjs1} it is easy to see that $\Exp (T)=0$. Thus $\Var T=\Exp (T^2)$. Now, using \eqref{eq: qjs3}, we see that
\begin{align*}
\sum_{\lambda\in \F_p} T_\lambda^2&=\sum_{x,y\in \F^*_p} \sum_{\lambda\in \F_p} \sigma(\lambda+r(x)) \sigma(\lambda+r(y))\\
&=\sum_{x,y\in \F^*_p} \Big(p\llbracket r(x)=r(y)\rrbracket-1\Big)\\
&=p-1+ p\sum_{x\neq y\in \F^*_p}\llbracket r(x)=r(y)\rrbracket
\end{align*}
since $p\sum_{x=y\in \F^*_p}\llbracket r(x)=r(y)\rrbracket=p(p-1)$ and $ \sum_{x,y\in \F^*_p} 1=(p-1)^2$. Therefore
\begin{align}\label{eq: wa}
\Var T=\frac{1}{p}\sum_{\lambda\in \F_p} T_\lambda^2=1-\frac{1}{p}+ \sum_{x\neq y\in \F^*_p}\llbracket r(x)=r(y)\rrbracket.
\end{align}

Let $\Sigma$ denote the latter double sum. For $x\neq y\in \F_p^*$, we have $\llbracket r(x)=r(y)\rrbracket=\llbracket q(x,y)=0\rrbracket$ where \[q(x,y)= x^2y^2\cdot\frac{r(x)-r(y)}{x-y}=x^2y^2-bxy-c(x+y).\]
Thus
\begin{align*}
\Sigma=\sum_{x\neq y\in \F^*_p}\llbracket  q(x,y)=0\rrbracket=\sum_{x,y\in \F^*_p}\llbracket  q(x,y)=0\rrbracket-\sum_{x\in \F^*_p}\llbracket  q(x,x)=0\rrbracket.
\end{align*}

(i) Assume that $c=0$; then $b\neq 0$. Over $\F_p^*$, the equation $q(x,y)=0$ amounts to $xy=b$; there are $p-1$ solutions $(x,y)\in \F_p^*\times \F_p^*$, and $1+\sigma(b)$ of them have the form $(x,x)$. Thus 
\begin{align*}
\Sigma=p-1-\big(1+\sigma(b)\big).
\end{align*}
In view of \eqref{eq: wa}, we have 
\begin{align}
\Var T=p-1-\sigma(b)-p^{-1}.
\end{align}

(ii) Assume that $c\neq 0$. The equation $q(x,x)=0$ amounts to $x^4-bx^2-2cx=0$, that is $x^3-bx-2c=0$ over $\F_p^*$. Since solutions to the latter equation are automatically non-zero, we have
\begin{align*}
\sum_{x\in \F^*_p}\llbracket  q(x,x)=0\rrbracket=\#\{x\in \F_p: x^3-bx-2c=0\}=n_3.
\end{align*}

For each $x\in \F^*_p$ the sum $\sum_{y\in \F^*_p}\llbracket q(x,y)=0\rrbracket$ counts the number of non-zero solutions $y\in \F_p$ to the quadratic equation $x^2y^2-(bx+c)y-cx=0$. That number is $1+\sigma\big(4cx^3+(bx+c)^2\big)$ since, once again, any solution $y$ is automatically non-zero. Thus
\begin{align*}
\sum_{x, y\in \F^*_p}\llbracket q(x,y)=0\rrbracket&=\sum_{x\in \F^*_p} \Big(1+\sigma\big(4cx^3+(bx+c)^2\big)\Big)\\
&=p-2+\sum_{x\in \F_p} \sigma\big(4cx^3+(bx+c)^2\big),
\end{align*}
by including the index value $x=0$ in the latter sum.

Combining our calculations, we see that
\begin{align*}
\Sigma=p-2-n_3+\sum_{x\in \F_p} \sigma\big(4cx^3+(bx+c)^2\big).
\end{align*}
Now \eqref{eq: wa} yields
\begin{align}
\Var T=p-1-n_3-p^{-1}+\sum_{x\in \F_p} \sigma\big(4cx^3+(bx+c)^2\big).
\end{align}
This completes the proof of Theorem~\ref{thm: A}.
\end{proof}

We now consider two examples in line with our main theme--that of using Jacobsthal sums to give explicit evaluations. 

\begin{ex}\label{ex: b0} Fix $c\in \F_p^*$. For the family of curves $\mathcal{C}_\lambda : y^2=x^3+\lambda x^2+c$, the variance of the point-count is given by
\begin{align*}
\Var \#\mathcal{C}=p-1-n_3-p^{-1}+\sum_{x\in \F_p} \sigma(4cx^3+c^2)
\end{align*}
where $n_3$ is the number of solutions in $\F_p$ to $x^3=2c$. The change of variable $x:=x/2$ in the latter sum leads to the following more convenient form:
\begin{align}\label{eq: ab0}
\Var \#\mathcal{C}=p-1-n_3-p^{-1}+\sigma(2c) \sum_{x\in \F_p} \sigma(x^3+2c).
\end{align}
Formula \eqref{eq: ab0} extends first- and second-moment computations of Miller \cite[Sec.13.2]{Mil1}, \cite{Mil2}, who considered the case $c=1$. 

The right-most sum in \eqref{eq: ab0} is $\psi_3(2c)$. With Lemma~\ref{lem: J2} at hand, we can turn \eqref{eq: ab0} into a more explicit formula as follows. When $p\equiv 2$ mod $3$ we have $\psi_3(2c)=0$, and $n_3=1$ for each $c\in \F_p^*$ since the cubing map $x\mapsto x^3$ is a permutation of $\F_p^*$. Thus
\begin{align}
\Var \#\mathcal{C}=p-2-p^{-1}.
\end{align}

When $p\equiv 1$ mod $3$, there are two cases according to whether $2c$ is, or is not, a cube in $\F^*_p$. We have, respectively, that $n_3=3$ or $n_3=0$, and
\begin{align}
\Var \#\mathcal{C}=\begin{cases}
p-4-p^{-1}+2A_3 & \textrm{ if } 2c\in (\F_p^*)^3, \\
p-1-p^{-1}-A_3\pm 3B_3 & \textrm{ if } 2c\notin (\F_p^*)^3.
\end{cases}
\end{align}
\end{ex}

\begin{ex} For the family of curves $\mathcal{C}_\lambda : y^2=x^3+\lambda x^2+6x+2$, the variance of the point-count  is given by
\begin{align*}
\Var \#\mathcal{C}=p-1-n_3-p^{-1}+\sum_{x\in \F_p} \sigma\big(8x^3+(6x+2)^2\big)
\end{align*}
where $n_3$ is the number of solutions in $\F_p$ to the equation $x^3-6x-4=0$. The cubic equation factors as $(x+2)(x^2-2x-2)=0$, so the number of solutions is $n_3=2+\sigma(3)$. On the other hand, a change of variable $x:=(x-1)/2$ yields
\begin{align*}
\sum_{x\in \F_p} \sigma\big(8x^3+(6x+2)^2\big)=\sum_{x\in \F_p} \sigma(x^3+6x^2-3x).
\end{align*}
The latter sum evaluates as $\sigma(-3)\psi_3(1)$, see \cite[Ch.5]{N}. Thus
\begin{align}
\Var \#\mathcal{C}&=p-3-\sigma(3)-p^{-1}+\sigma(-3)\psi_3(1)
\end{align}
and, thanks to Lemma~\ref{lem: J2}, we conclude that
\begin{align}
\Var \#\mathcal{C}=\begin{cases}
p-3-\sigma(3)-p^{-1}+2A_3 & \textrm{ if } p\equiv 1 \textrm{ mod } 3, \\
p-3-\sigma(3)-p^{-1} & \textrm{ if } p\equiv 2 \textrm{ mod } 3.
\end{cases}
\end{align}
\end{ex}

\medskip
\section{The family $\mathcal{C}_\lambda : y^2=x^3+bx+c+\lambda (x^2-x)$}\label{sec: S}
We start by pointing out a duality underlying the dependence on $b$ and $c$ for the family $y^2=x^3+bx+c+\lambda (x^2-x)$, which we temporarily denote by $\mathcal{C}(b,c)$. The change of variable $x:=1-x$ transforms $y^2=x^3+bx+c+\lambda (x^2-x)$ into $-y^2=x^3+bx-(b+c+1)-(\lambda+3) (x^2-x)$. We thus obtain a twist of the family $\mathcal{C}(b,-(b+c+1))$. The upshot is that the variance of the number of points along the two families, $\mathcal{C}(b,c)$ and $\mathcal{C}(b,-(b+c+1))$, is the same. We note that the correspondence $(b,c)\leftrightarrow (b,-(b+c+1))$ is an involution, whose fixed points are all pairs $(b,c)$ satisfying $b+1=-2c$.

The above duality offers a preliminary insight into the pairs $(b,c)$ that are tabulated in Theorem~\ref{thm: D}. The double pairs in the top half of the table are dual pairs. The pairs in the bottom half are fixed under duality.

\begin{proof}[Proof of Theorem~\ref{thm: D}] Fix $b,c\in \F_p$. For $\lambda\in \F_p$, set
\begin{align}\label{eq: yf1}
S_\lambda= \sum_{x\in \F_p} \sigma\big(x^3+bx+c+\lambda (x^2-x)\big).
\end{align}
As before, our aim is to compute the variance of $(S_\lambda)_{\lambda\in \F_p}$. We first write
\begin{align*}
S_\lambda= \sigma(c)+\sigma(1+b+c)+\sum_{x\neq 0,1} \sigma\big(x^3+bx+c+\lambda (x^2-x)\big).
\end{align*}
On $\F_p\setminus\{0,1\}$, the map $x\mapsto x/(x-1)$ is a permutation so we may use it as a change of variable in the latter sum. The polynomial argument becomes
\begin{align*}
\frac{x^3}{(x-1)^3}+\frac{bx}{x-1}+c+\lambda \bigg(\frac{x^2}{(x-1)^2}-\frac{x}{x-1}\bigg)=\frac{x}{(x-1)^2}\big(\lambda+r(x)\big)
\end{align*} 
where 
\begin{align}
r(x)=\frac{x^3+bx(x-1)^2+c(x-1)^3}{x(x-1)}.
\end{align}
Thus
\begin{align*}
S_\lambda=\sigma(c)+\sigma(1+b+c)+\sum_{x\neq 0,1} \sigma(x)\sigma\big(\lambda+r(x)\big).
\end{align*}
By translation invariance, $\Var T=\Var S$ where $T=(T_\lambda)_{\lambda\in \F_p}$ is given by
\begin{align}
T_\lambda=\sum_{x\neq 0,1} \sigma(x)\sigma\big(\lambda+r(x)\big).
\end{align}
It is easy to see, using \eqref{eq: qjs1}, that $\Exp (T)=0$. Thus $\Var T=\Exp (T^2)$.

Now, using \eqref{eq: qjs3}, we get
\begin{align*}
\sum_{\lambda\in \F_p} T_\lambda^2&= \sum_{x,y\neq 0,1}\sum_{\lambda\in \F_p}  \sigma(xy)\sigma\big(\lambda+r(x)\big)\sigma\big(\lambda+r(y)\big)\\
&=\sum_{x,y\neq 0,1} \sigma(xy)\Big(p\llbracket r(x)=r(y)\rrbracket-1\Big).
\end{align*}
Note that $\sum_{x,y\neq 0,1} \sigma(xy)=\big(\sum_{x\neq 0,1} \sigma(x)\big)^2=1$. We obtain
\begin{align*}
\Var T=\frac{1}{p}\sum_{\lambda\in \F_p} T_\lambda^2= -\frac{1}{p}+\sum_{x,y\neq 0,1} \sigma(xy)\llbracket r(x)=r(y)\rrbracket.
\end{align*}
The diagonal terms, namely those accounting for $x=y\in \F_p\setminus\{0,1\}$, contribute $p-2$ to the double sum. When $x, y\in \F_p\setminus\{0,1\}$ are distinct, we have $\llbracket r(x)=r(y)\rrbracket=\llbracket q(x,y)=0\rrbracket$ where
\begin{align*}
q(x,y)=x(x-1)y(y-1)\cdot \frac{r(x)-r(y)}{x-y}
\end{align*}
that is,
\begin{align*}
q(x,y)=xy(xy-x-y)+bxy(x-1)(y-1)+c(x-1)(y-1)(xy-1).
\end{align*}
Thus far, we have
\begin{align}
\Var T= p-2-p^{-1}+\Sigma
\end{align}
where
\begin{align}
\Sigma= \sum_{x\neq y \in \F_p\setminus\{0,1\}} \sigma(xy)\llbracket q(x,y)=0\rrbracket.
\end{align}

We now turn to the computation of $\Sigma$, which is what $\Sigma(b,c)$ stands for in the statement of Theorem~\ref{thm: D}. As usual, we reinstate the diagonal contribution by writing
\begin{align*}
\Sigma=\sum_{x, y \neq 0,1} \sigma(xy)\llbracket q(x,y)=0\rrbracket-\sum_{x \neq 0,1} \llbracket q(x,x)=0\rrbracket.
\end{align*}
The first double sum is not affected by adding the index values $x=1$ and $y=1$. For $q(x,1)=-x\neq 0$ whenever $x\in \F_p^*$; likewise, $q(1,y)\neq 0$ whenever $y\in \F_p^*$. The second sum counts the number 
\begin{align}
N=\#\{x\in \F_p\setminus\{0,1\}: q(x,x)=0\}.
\end{align}
To summarize, we have
\begin{align}
\Sigma=-N+\sum_{x, y \in \F_p^*} \sigma(xy)\llbracket q(x,y)=0\rrbracket.
\end{align}
In the latter double sum, we make the change of variables $x:=x/y$ for each $y\in \F_p^*$; we arrive at
\begin{align}
\Sigma&=-N+\sum_{x \in \F_p^*} \sigma(x)\sum_{y \in \F_p^*} \llbracket q(x/y,y)=0\rrbracket.
\end{align}
For each $x\in \F_p^*$, the inner sum counts 
\begin{align}
N(x)=\#\{y \in \F_p^*: q(x/y,y)=0\}.
\end{align}

We break down the rest of the argument into two main cases, $c=0$ respectively $b+1=-2c\neq 0$. The first case yields the results collected in the top half of the table of Theorem~\ref{thm: D}, but it also covers the self-dual pair $(b,c)=(-1,0)$. The second case yields the last two pairs from the table.

\medskip
\noindent \textbf{Case $c=0$.} Then
\begin{align}
q(x,y)=xy(xy-x-y)+bxy(x-1)(y-1)
\end{align}

If $b=-1$ we have, very simply, $q(x,y)=-xy$. We see that $N=0$ and $N(x)=0$ for each $x\in \F_p^*$. Thus, for $(b,c)=(-1,0)$, we have
\begin{align*}
\Sigma=0.
\end{align*}

If $b=0$ we have, also quite simply, $q(x,y)=xy(xy-x-y)$. For each $x\in \F_p^*$, the equation $q(x/y,y)=0$ amounts to $y+x/y=x$; this quadratic equation has $N(x)=1+\sigma(x^2-4x)$ solutions $y\in \F_p^*$. The equation $q(x,x)=0$ amounts to $x^2-2x=0$; this has just one solution $x\neq 0, 1$, so $N=1$. Thus, for $(b,c)=(0,0)$, we have
\begin{align*} 
\Sigma&= -N+\sum_{x \in \F_p^*} \sigma(x)N(x)\\
&=-1+\sum_{x \in \F_p^*} \sigma(x)\big(1+\sigma(x^2-4x)\big)\\
&=-1+\sum_{x \in \F_p^*} \sigma(x-4)=-1-\sigma(-1).
\end{align*}

Now assume that $b\neq 0, -1$. We write $q(x,y)=xy\big((b+1)(x-1)(y-1)-1\big)$; for $x,y\in \F_p^*$, the equation $q(x,y)=0$ amounts to $(b+1)(x-1)(y-1)=1$.

On the one hand, the equation $q(x,x)=0$ amounts to $(b+1)(x-1)^2=1$. This quadratic has $1+\sigma(b+1)$ solutions in $\F_p$, and we note that $0$ or $1$ cannot be solutions. We thus have $N=1+\sigma(b+1)$. 

On the other hand, for each $x\in \F_p^*$ the equation $q(x/y,y)=0$ amounts to 
\[y+\frac{x}{y}=x+\frac{b}{b+1}.\] 
We momentarily put $d:=b/(b+1)$. Then $N(x)=1+\sigma((x+d)^2-4x)$. We compute
\begin{align*}
\Sigma&= -N+\sum_{x \in \F_p^*} \sigma(x)N(x)\\
&= -1-\sigma(b+1)+\sum_{x \in \F_p^*} \sigma(x)\big(1+\sigma((x+d)^2-4x)\big)\\
&=-1-\sigma(b+1)+\sum_{x \in \F_p} \sigma(x)\sigma\big((x+d)^2-4x\big)\\
&=-1-\sigma(b+1)+\sigma(d) \sum_{x \in \F_p} \sigma(x)\sigma\Big((x+1)^2-\frac{4}{d}\cdot x\Big)
\end{align*}
where, in the last step, we made the change of variables $x:=dx$. Recall that $d\neq 0$. Reverting from $d$ to $b$, and rearranging the quadratic argument, we deduce the formula
\begin{align}
\Sigma=-1-\sigma(b+1)+\sigma(b(b+1)) \sum_{x \in \F_p} \sigma\Big(x^3-\frac{2b+4}{b}\cdot x^2+x\Big).
\end{align}

At this point, we specialize $b$ so as to be able to evaluate the latter sum in terms of the Jacobsthal sum $\varphi_2(1)$.

Let $b=-2$, so $(b,c)=(-2,0)$. We then get
\begin{align*}
\Sigma&=-1-\sigma(-1)+\sigma(2) \sum_{x \in \F_p} \sigma(x^3+x)\\
&=-1-\sigma(-1)+\sigma(2) \varphi_2(1).
\end{align*}

Let $b=1$, so $(b,c)=(1,0)$. We then get
\begin{align*}
\Sigma=-1-\sigma(2)+\sigma(2) \sum_{x \in \F_p} \sigma(x^3-6x^2+x).
\end{align*}
The latter sum evaluates as $\sigma(2)\varphi_2(1)$, see \cite[Ch.5]{N}. We conclude that
\begin{align*}
\Sigma=-1-\sigma(2)+\varphi_2(1).
\end{align*} 

Let $b=-1/2$, so $(b,c)=(-1/2,0)$. We then get
\begin{align*}
\Sigma&=-1-\sigma(2)+\sigma(-1) \sum_{x \in \F_p} \sigma(x^3+6x^2+x)\\
&=-1-\sigma(2)+\sum_{x \in \F_p} \sigma(x^3-6x^2+x)\\
&=-1-\sigma(2)+\sigma(2)\varphi_2(1).
\end{align*} 
In the second step we changed the variable $x:=-x$, and then we used the same evaluation as before.

\medskip
\noindent \textbf{Case $b+1=-2c\neq 0$.} Using $b=-(2c+1)$, we compute
\begin{align}
q(x,y)=-xy-c(x-1)(y-1)(xy+1).
\end{align}

Next, we specialize to two values of $c$. In the first case, $c=1$, the algebra runs very smoothly. In the second case, $c=-1$, the algebra is quite unruly.

Let $c=1$, so $(b,c)=(-3,1)$. In this case we can write $q(x,y)=-(xy-x+1)(xy-y+1)$. In particular, the equation $q(x,x)=0$ amounts to $(x^2-x+1)^2=0$; this has $1+\sigma(-3)$ solutions $x\in \F_p$, which automatically satisfy $x\neq 0,1$. Thus $N=1+\sigma(-3)$. 

For $x,y\in \F_p^*$ the equation $q(x/y,y)=0$ can be brought to $\big((x+1)y-x\big)\big(y-(x+1)\big)=0$. There are no solutions $y\in \F_p^*$ when $x=-1$. When $x\neq -1$, the solutions are $y=x/(x+1)$ and $y=x+1$ ; both are non-zero, and they agree precisely when $x^2+x+1=0$. This means that, for $x\neq -1$, we have $N(x)=2-\llbracket x^2+x+1=0\rrbracket$.

We obtain
\begin{align*}
\Sigma&= -N+\sum_{x \in \F_p^*} \sigma(x)N(x)\\
&=-1-\sigma(-3)+\sum_{x\neq 0,-1} \sigma(x)\big(2-\llbracket x^2+x+1=0\rrbracket\big)\\
&=-1-\sigma(-3)-2\sigma(-1)-\sum_{x\in \F_p} \sigma(x)\llbracket x^2+x+1=0\rrbracket.
\end{align*}

The equation $x^2+x+1=0$ has $1+\sigma(-3)$ solutions. By rewriting it as $(x+1)^2=x$, we see that each solution is a square in $\F_p^*$. We conclude that
\begin{align*}
\Sigma=-1-\sigma(-3)-2\sigma(-1)-\big(1+\sigma(-3)\big)=-2\big(1+\sigma(-3)+\sigma(-1)\big).
\end{align*}

Let $c=-1$, so $(b,c)=(1,-1)$. For $x,y\in \F_p^*$ the equation $q(x/y,y)=0$ turns into $(x+1)y^2-(x^2+x+1)y+x(x+1)=0$. There are no solutions $y\in \F_p^*$ when $x=-1$. For $x\neq -1$, the quadratic equation has $N(x)=1+\sigma(\Delta(x))$ solutions $y\in \F_p^*$, where 
\[\Delta(x)=(x^2+x+1)^2-4x(x+1)^2=x^4-2x^3-5x^2-2x+1.\]
We then have
\begin{align*}
\Sigma&= -N+\sum_{x \in \F_p^*} \sigma(x)N(x)\\
&=-N+\sum_{x \neq 0,-1} \sigma(x)\big(1+\sigma(\Delta(x))\big)\\
&=-N- \sigma(-1)\big(1+\sigma(\Delta(-1))\big)+\sum_{x \in \F_p} \sigma(x)\big(1+\sigma(\Delta(x))\big)\\
&=-N-2\sigma(-1)+\sum_{x \in \F_p} \sigma(x^5-2x^4-5x^3-2x^2+x).
\end{align*}

We now use the following result \cite[Thm.5.20]{N}: the quadratic character sum with palindromic quintic argument
\[\sum_{x \in \F_p} \sigma(x^5+\alpha x^4+\beta x^3+\alpha x^2+x)\]
equals a combination of quadratic character sums with cubic arguments, namely
\[\sigma(\alpha+4) \varrho\bigg(\frac{\beta+2\alpha+2}{(\alpha+4)^2}\bigg)+\sigma(\alpha-4)\varrho\bigg(\frac{\beta-2\alpha+2}{(\alpha-4)^2}\bigg)\]
as long as $\alpha \neq \pm 4$. Here, we use the \emph{$\varrho$-sums} defined by
\[\varrho(c)=\sum_{x\in \F_p} \sigma(x^3+x^2+cx).\]
In our case $\alpha=-2$ and $\beta=-5$ so
\[\sum_{x \in \F_p} \sigma(x^5-2x^4-5x^3-2x^2+x)=\sigma(2)\varrho\Big(\frac{-7}{4}\Big)+\sigma(-6)\varrho\Big(\frac{1}{36}\Big).\]
The $\varrho$-sums satisfy the twisted symmetry \cite[Thm.5.3]{N}
\[\varrho \bigg(\frac{1}{4}-c\bigg)=\sigma(-2)\varrho(c).\]
Then $\varrho(-7/4)=\sigma(-2)\varrho(2)$. Also, $\varrho(1/36)=\sigma(-3)\varphi_2(1)$ \cite[Cor.5.4]{N}. Thus
\[\sum_{x \in \F_p} \sigma(x^5-2x^4-5x^3-2x^2+x)=\sigma(-1)\varrho(2)+\sigma(2)\varphi_2(1).\]
Summarizing, we have
\begin{align*}
\Sigma=-N-2\sigma(-1)+\sigma(-1)\varrho(2)+\sigma(2)\varphi_2(1)
\end{align*}
where $N=n_4$ is the number of solutions $x\in \F_p$ to the equation $q(x,x)=0$, that is, $x^4-2x^3+x^2-2x+1=0$. 
\end{proof}



\begin{thebibliography}{88}
\bibitem{B} B.J. Birch: \emph{How the number of points of an elliptic curve over a fixed prime field varies}, J. London Math. Soc. 43 (1968), 57--60

\bibitem{HM} S. He, J. McLaughlin: \emph{Some remarks on the number of points on elliptic curves over finite prime field}, Bull. Austral. Math. Soc. 75 (2007), no. 1, 135--149

\bibitem{H} H. Hopf: \emph{\"Uber die Verteilung quadratischer Reste}, Math. Z. 32 (1930), no. 1, 222--231

\bibitem{J0} E. Jacobsthal: \emph{Anwendungen einer Formel aus der Theorie der quadratischen Reste}, Dissertation, Berlin, 1906

\bibitem{J} E. Jacobsthal: \emph{\"Uber die Darstellung der Primzahlen der Form $4n+1$ als Summe zweier Quadrate}, J. Reine Angew. Math. 132 (1907), 238--246


\bibitem{Mil1} S.J. Miller: \emph{$1$- and $2$-level densities for families of elliptic curves: evidence for the underlying group symmetries}, PhD thesis Princeton University (2002)

\bibitem{Mil2} S.J. Miller: \emph{Variation in the number of points on elliptic curves and applications to excess rank}, C. R. Math. Acad. Sci. Soc. R. Can. 27 (2005), no. 4, 111--120

\bibitem{N} B. Nica: \emph{Jacobsthal Sums}, Monographs in Number Theory, World Scientific (forthcoming)

\bibitem{Y} M. Yamauchi: \emph{Some identities on the character sum containing $x(x-1)(x-\lambda)$}, Nagoya Math. J. 42 (1971), 109--113
\end{thebibliography}
\end{document}